%% file: Fullversion.tex
\documentclass[11pt]{article}
\usepackage[twoside,a4paper,margin=1in]{geometry}

\usepackage{graphicx,url,amsmath,amsthm,amsfonts,amssymb,subfigure,bbm,bm}

\usepackage{hyperref,algorithmic}

\usepackage{color}
\usepackage{microtype}
\usepackage{verbatim}
\usepackage{wrapfig}
\usepackage{diagrams}

\usepackage{thm-restate}
\usepackage{framed}
\usepackage{tikz-cd}
\usepackage{wrapfig}

\newcommand{\executeiffilenewer}[3]{%
\ifnum\pdfstrcmp{\pdffilemoddate{#1}}%
{\pdffilemoddate{#2}}>0%
{\immediate\write18{#3}}\fi%
}

\def\note#1{\GenericWarning{}%
  {AUTHOR WARNING: Unresolved annotation}%

}
\definecolor{blueblack}{rgb}{0,0,.7}
\newcommand{\emphdef}[1]{%
  \textcolor{blueblack}{%
    \textbf{\emph{#1}}%
  }%
}

\DeclareMathOperator{\Ima}{Im}

\newcommand{\R}{\mathbb{R}}

\newcommand{\NP}{\textbf{NP}}

 \newtheorem{theorem}{Theorem}
 \newtheorem{lemma}[theorem]{Lemma}

 \newtheorem{proposition}[theorem]{Proposition}

\title{Finding non-orientable surfaces in $3$-manifolds\thanks{The second author has received funding from the People Programme
(Marie Curie Actions) of the European Union's Seventh Framework Programme
(FP7/2007-2013) under REA grant agreement n° [291734].
The first author is supported by the Australian Research Council
(project DP140104246).}}%
\author{Benjamin A. Burton\thanks{School of Mathematics and Physics, The University of Queensland, Brisbane, Australia, \protect\url{bab@maths.uq.edu.au}}
\and Arnaud de Mesmay\thanks{CNRS, Gipsa-Lab, Grenoble, France,
\protect\url{arnaud.de-mesmay@gipsa-lab.fr}}
\and Uli Wagner\thanks{IST Austria, Klosterneuburg, Austria, \protect\url{uli@ist.ac.at}}
}

\begin{document}

\maketitle 

\begin{abstract}
  We investigate the complexity of finding an embedded non-orientable surface of
  Euler genus $g$ in a triangulated $3$-manifold. This problem occurs
  both as a natural question in low-dimensional topology, and as a
  first non-trivial instance of embeddability of complexes into
  $3$-manifolds.

  We prove that the problem is NP-hard, thus adding to the relatively
  few hardness results that are currently known in 3-manifold
  topology. In addition, we show that the problem lies in NP when the
  Euler genus g is odd, and we give an explicit algorithm in this
  case.
\end{abstract}

\section{Introduction}

Since the foundational work of Haken~\cite{h-tnik-61} on unknot
recognition, the past decades have witnessed a flurry of algorithms
designed to solve decision problems in low-dimensional topology. Many
of these results rely on the framework of \textit{normal surfaces},
which provide a compact and algebraic way to analyze and enumerate the
noteworthy surfaces embedded in a 3--manifold. In a nutshell, many
low-dimensional problems can be seen as an instance of the following
(intentionally vague) question, which encompasses the class of problems
that normal surface theory has been designed to solve:

\begin{framed}
  \textsc{Generic 3-manifold problem}
  
  \textbf{Input:} A $3$-manifold $M$.
  
  \textbf{Question:} Find an ``interesting'' surface in $M$.
\end{framed}

For example, for unknot recognition~\cite{hlp-ccklp-99}, one
triangulates the complement of the knot and looks for a spanning disk
that the knot bounds, while for knot genus~\cite{aht-cckgs-06}, one
looks for a Seifert surface of minimal genus instead. To solve
$3$-sphere recognition~\cite{r-ar3s-95,t-tprps-94}, one looks for a
maximal collection of stable and unstable spheres~\cite{h-wans-12}.
Prime decomposition~\cite{k-gfdm-29} and JSJ
decomposition~\cite{js-sfs3m-79,j-he3mb-79} work by finding embedded
spheres or tori in a $3$-manifold -- note that these decompositions are
the first steps to test homeomorphism of
$3$-manifolds~\cite{k-ah3mcg-15}, which is often considered a holy grail
of computational $3$-manifold theory. Other examples include the computation of Heegard genus (and Heegard splittings)~\cite{l-hsmlwc-07,l-adhga3-11}, determining whether a manifold is Haken~\cite{jo-ad3mhm-84} or the crosscap number of a knot~\cite{bo-ccnkip-12}.

In this work, we investigate one of the most natural instances of this generic problem: since every $3$-manifold contains every orientable surface, these (at least without further restrictions) can be considered uninteresting, and therefore the first non-trivial question is the following:

\begin{framed}
  \textsc{Non Orientable Surface Embeddability}

\textbf{Input}: An integer $g$ and a triangulation of a closed $3$-manifold $M$.

\textbf{Question}: Does the non-orientable surface of Euler genus $g$ embed into $M$?
\end{framed}

This question is not just a toy problem for computational
$3$-manifold theory: non-orientable surfaces embedded in a
$3$-manifold provide structural informations about it. Following the
foundational article of Bredon and Wood~\cite{bw-noso3m-69}
classifying non-orientable surfaces in lens spaces and surfaces
bundles, many works have been devoted to this study for specific
$3$-manifolds or specific surfaces (see for
example~\cite{e-nos3m-92,ih-nofsls-09,k-s3makb-78,lrs-nshc-15,r-issfs-96,r-3mffgc-79,r-nsnh3m-82}). Our
work complements these by investigating the complexity of finding
non-orientable surfaces in the most general setting.

Another motivation for studying this question comes from the higher
dimensional analogues of graph embeddings. Graphs generalize naturally
to simplicial complexes, and several recent efforts have been made to
study higher dimensional versions of the classical notions of planar
or surface-embedded graphs~\cite{mtw-hescr-11, mstw-e3sd-14,
  w-mreh-11}, see also Skopenkov~\cite{s-ekmes-07} for some mathematical background. In particular, Matou\v{s}ek, Sedgwick,
Tancer and Wagner~\cite{mstw-e3sd-14} recently showed that testing
whether a given $2$-complex embeds in $\mathbb{R}^3$ is
decidable -- the main algorithmic machinery
  underlying this result is yet another instance of the generic
  $3$-manifold problem! In their paper, they ask what is the
complexity of this problem for embeddings into other $3$-manifolds (as
opposed to $\R^3$), and since a
non-orientable surface is a particular simple instance of a $2$-complex, \textsc{Non Orientable Surface Embeddability} is the first
problem to investigate in this direction.

\subparagraph*{Our results.} Our first result is a proof of hardness.

\begin{restatable}{theorem}{Tarnaud}\label{T:arnaud}
The problem \textsc{Non Orientable Surface Embeddability} is \NP-hard.
\end{restatable}

As an immediate corollary, it is thus \NP-hard to decide,
given a $2$-complex $K$ and a $3$-manifold $M$, whether $K$ embeds
into $M$\footnote{On the other hand this problem is not even known to
  be decidable. This places it in the same complexity limbo as testing
  embeddability of $2$-complexes into
  $\mathbb{R}^4$~\cite{mtw-hescr-11}.}. This might not come as
surprise: this is a higher-dimensional version of \textsc{Graph
  Genus}, which is already known to be
\NP-hard~\cite{t-ggpnc-89}. However, we would like to emphasize that
non-orientable surfaces are among the simplest possible instances of
$2$-complexes, namely $2$-manifolds, and by contrast deciding whether
a $1$-manifold, i.e., a circle graph, embeds on a surface is
trivial. Furthermore, hardness results are well known to be elusive in
$3$-manifold topology, where iconic problems such as unknot
recognition and 3-sphere recognition lie in \NP $\cap$
co-\NP~\cite{hk-nrcr3-12,hlp-ccklp-99,l-ecktn-16,s-srlnp-11}\footnote{Note
  that the proof of co-\NP \ membership for 3-sphere
  recognition~\cite{hk-nrcr3-12} assumes the Generalized Riemann
  Hypothesis.}, and nothing is known for most other problems, the
notable exception being \textsc{3-Manifold Knot
  Genus}~\cite{aht-cckgs-06} which is known to be
\NP-complete\footnote{In parallel to this work, new \NP-hardness
  results have appeared very recently, for \textsc{Heegaard
    genus}~\cite{bds-chgnph-16} and for the \textsc{Sublink} problem
  and the \textsc{Upper bound for the Thurston complexity of an
    unoriented classical link.}~\cite{l-chpl3m-16}}.  Our result can
be seen as a hint that many three-dimensional problems are hard when
the description of a $3$-manifold is part of the input.

The proof of Theorem~\ref{T:arnaud} starts similarly to the
aforementioned one for \textsc{3-Manifold Knot Genus} by Agol, Hass
and Thurston: the idea is to encode an instance of
\textsc{One-in-Three SAT} within the embeddability of a non-orientable
surface inside a $2$-complex. This complex is then turned into a
$3$-manifold by a \textit{thickening} step and a \textit{doubling}
step. A key argument in the proof of the reduction of Agol, Hass and
Thurston revolves around computing a \textit{topological degree},
which is trivial in the case of knot genus. It turns out that this
computation still works but is significantly harder in our setting,
and this is the main technical hurdle in our case, for which we need
to introduce (co-)homological ingredients.

\smallskip

Our second result provides an algorithm for this problem, provided that $g$ is odd, proving that it is also in \NP.

\begin{restatable}{theorem}{Tnormal}\label{T:normal}
Let $g$ be an odd positive integer and $M$ a triangulation of a
3-manifold. The problem \textsc{Odd Non Orientable Surface Embeddability} of testing whether $M$ contains a non-orientable surface of Euler genus $g$ is in \NP.
\end{restatable}

Observing that in the reduction involved in the proof of
Theorem~\ref{T:arnaud}, the non-orientable surface that we use has odd
Euler genus, we immediately obtain as a corollary that \textsc{Odd Non
  Orientable Surface Embeddability} is \NP-complete.

As is the case with many problems in low-dimensional topology, proving
membership in \NP \, is not as trivial as most computer scientists might
be accustomed to. As an illustration, our techniques fail for even
values of $g$, and in these cases the problem is not even known to be
decidable. A particularity of our proof is to leverage on the recent
simplifications due to Burton~\cite{b-nac3mt-14} of the \textit{crushing} procedure of Jaco and
Rubinstein~\cite{jr-0et3m-03} to reduce the problem to the case of an
irreducible $3$-manifold. Then our proof relies on normal surface
theory.

\section{Preliminaries}\label{S:prelim}

We only recall here the definitions of the basic objects which we
investigate in this article. The technical tools used in the proofs
will be introduced when needed, and in general we will assume that the
reader is familiar with the basic concepts of algebraic topology, as
explained for example in Hatcher~\cite{h-at-02}.

A \emphdef{surface} (resp.\ a \emphdef{surface with boundary}) is a
topological space which is locally homeomorphic to the plane
(resp.\ locally homeomorphic to the plane or the half-plane). By the
theorem of classification of surfaces, these are classified up to
homeomorphism by their \emphdef{orientability} and their
\emphdef{genus} (and the number of boundaries if there are any). Since
we will deal frequently with non-orientable surfaces, when we use the word
genus we actually mean \emphdef{Euler genus}, sometimes also called
\emphdef{non-orientable genus}, which equals twice the usual genus for
orientable surfaces. In particular, any surface with odd genus is
non-orientable.

A \emphdef{$3$-manifold} (resp.\ a \emphdef{3-manifold with boundary})
is a topological space which is locally homeomorphic to
$\mathbb{R}^3$, resp.\ to $\mathbb{R}^3$ or the half-space $\mathbb{R}^3_{|x \geq 0}$. To
be consistent with the literature in low-dimensional topology, we will
describe $3$-manifolds not with simplicial complexes, but with the
looser concept of (generalized) $\emphdef{triangulations}$, which are
defined as a collection of $n$ abstract tetrahedra, all of whose $4n$
faces are glued together in pairs. In particular, we allow two faces
of the same tetrahedron to be identified. Note that the underlying
topological space may not be a $3$-manifold, but if each vertex of the
tetrahedra has a neighborhood homeomorphic to $\mathbb{R}^3$ and no
edge is identified to itself in the reverse direction, we obtain a
$3$-manifold~\cite{m-as3mtth-52}.

A \emphdef{simplicial complex} $K$ is a set of simplices such that any face from a simplex in $K$ is also in $K$, and the intersection of two simplices $s_1$ and $s_2$ of $K$ is either empty or a face of both $s_1$ and $s_2$. In this article, we will only deal with $2$-dimensional simplicial complexes, which are simplicial complexes where the maximal dimension of the simplices is $2$ -- these can be safely thought of as triangles glued together along their subfaces.

\section{Hardness result}\label{S:arnaud}

In this section we prove the following theorem.

\Tarnaud*

Our reduction is inspired by the proof of Agol, Hass and
Thurston~\cite{aht-cckgs-06} that \textsc{Knot Genus in 3-manifolds}
is \NP-hard. While the idea of the reduction is similar, the proof of
its correctness is considerably more tricky. We use a reduction from
the \NP-complete~\cite{s-csp-78} problem \textsc{One-in-Three SAT},
which we first recall. It is defined in terms of literals (boolean
variables or their negations) gathered in clauses consisting of three
literals.

\begin{framed}
\textsc{One-in-Three SAT}

\textbf{Input}: A set of variables $U$ and a set of clauses $C$ over $U$ such that each clause contains exactly $3$ literals.

\textbf{Question}: Does there exist a truth assignment for $U$ such that each clause in $C$ has exactly one true literal? 
\end{framed}

Starting from an
instance $I$ of \textsc{One-in-Three SAT}, we will build a non-orientable
surface $S$ and a $3$-manifold $M$ such that $S$ embeds into $M$ if
and only if $I$ is satisfiable.

\subsection{The gadget}\label{S:gadget}

Let $I$ be an instance of \textsc{One-in-Three SAT}, consisting of a
set $U=\{u_1, \ldots , u_n\}$ of variables and a set $C=\{c_1, \ldots,
c_m\}$ of clauses. The surface $S$ is taken to be the non-orientable
surface of Euler genus $2m+2n+1$. The construction of $M$ is more
intricate, and follows somewhat the construction of the $3$-manifold
of Agol, Hass and Thurston, but with a M\"obius band glued on the
boundary. We build $M$ in three steps.

\begin{enumerate}
\item We first build a $2$-dimensional complex $K$.
\item We \textit{thicken} $K$ into a $3$-manifold $N$ with boundary.
\item We \textit{double} $N$, that is, we glue two copies of $N$ along
  their common boundary to obtain $M$.
\end{enumerate}

We first describe how these spaces are defined topologically, and
address in Lemma~\ref{L:complexity} the issue of computing an actual
triangulation of $M$.

  \textbf{First step.} The complex $K$ is obtained in the following
  way. We start with a projective plane $P$ with $n+m$ boundary
  curves, which we label by $u_1, \ldots, u_n, c_1, \ldots, c_m$. Let
  us denote by $k_i$ the number of times that the variable $u_i$
  appears in the collection of clauses $K$, and $\bar{k_i}$ the number
  of times that the negation of $u_i$ appears. Fix an
  orientation\footnote{Since $P$ is not orientable, this is of course
    not well-defined. We mean an orientation ``in the northern
    hemisphere'' of $P$ in Figure~\ref{F:BranchSurf}. Up to
    homeomorphism, it does not change anything, but this will be
    useful for the surgery arguments used throughout the proof.} of
  the boundary curves as in Figure~\ref{F:BranchSurf}. When gluing
  surfaces along curves, we will always use orientation-reversing
  homeomorphisms.

For $i=1, \ldots, n$, let $F_{u_i}$ and $F_{\bar{u_i}}$ be genus one
surfaces with $k_{i}+1$ and $\bar{k_i}+1$ boundaries. For each $i$,
one boundary curve from the surface $F_{u_i}$ is identified to
$u_i$. The remaining $k_i$ boundary components are identified with
each of the curves $c_j$ such that $u_i$ appears in $c_j$. Similarly,
$F_{\bar{u_i}}$ is attached to $\bar{u_i}$ and to every curve $c_j$
for which $\bar{u_i}$ appears in $c_j$. In the end, three surfaces are
attached along each $u_i$ ($F_{u_i}$, $F_{\bar{u_i}}$ and $P$), and
four surfaces are attached along each $c_i$ ($P$ and the surfaces
corresponding to the three litterals in $c_i$). We call the curves
$u_1, \ldots, u_n, c_1, \ldots c_m$ the \emphdef{branching cycles} of
$K$, and we refer to Figure~\ref{F:BranchSurf} for an illustration.

\begin{figure}
\centering
\def\svgwidth{7cm}
\executeiffilenewer{BranchSurf.svg}{BranchSurf.pdf}%
{inkscape -z -D --file=fig/BranchSurf.svg %
--export-pdf=fig/BranchSurf.pdf --export-latex}%
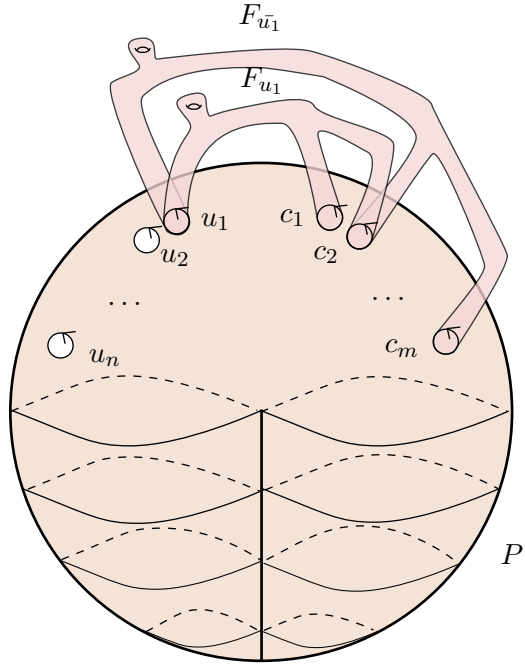%

\caption{The projective plane $P$ with its $n+m$ boundary curves, and
  examples of surfaces $F_{u_1}$ and $F_{\bar{u_1}}$ glued to clauses
  containing $u_1$, respectively $\bar{u_1}$.}
\label{F:BranchSurf}
\end{figure}

\textbf{Second step.} A $3$-manifold $M$ is a \textit{thickening} of a
2-dimensional complex $K$ if there exists an embedding $f:K \rightarrow
M$ such that $M$ is a regular neighborhood of $f(K)$. Intuitively, a
thickening corresponds to the idea of growing a
$3$-dimensional neighborhood around a $2$-complex, but some care is
needed, as not every 2-complex is thickenable -- see for example
Skopenkov~\cite{s-gntt2-95} for more details on this operation.

In our case though, the complex $K$ is always thickenable, and the
process is exactly the same as in the proof of Agol, Hass and
Thurston. When $K$ is locally a surface, the thickening just amounts
to taking a product with a small interval
(Figure~\ref{F:thicken}a.). Therefore, to define a thickening of $K$
it suffices to describe how to thicken around its singular points,
which by construction are the branching curves $u_1 \ldots u_n, c_1
\ldots c_m$. If $F_1, \ldots F_k$ are the surfaces adjacent to a
boundary curve, one can just pick a permutation of the surfaces around
the curve and thicken the complex following this permutation, as in
Figure~\ref{F:thicken}b.

This is akin to the fact that an embedding of a graph on a surface is
described by a permutation of the edges around each vertex. Applying
this construction for every boundary curve, we obtain a $3$-manifold
with boundary $N$ since every point close to the branching circles has
now a neighborhood locally homeomorphic to $\mathbb{R}^3$.

\begin{figure}
\centering
\def\svgwidth{11cm}
\executeiffilenewer{thicken1.svg}{thicken1.pdf}%
{inkscape -z -D --file=fig/thicken1.svg %
--export-pdf=fig/thicken1.pdf --export-latex}%
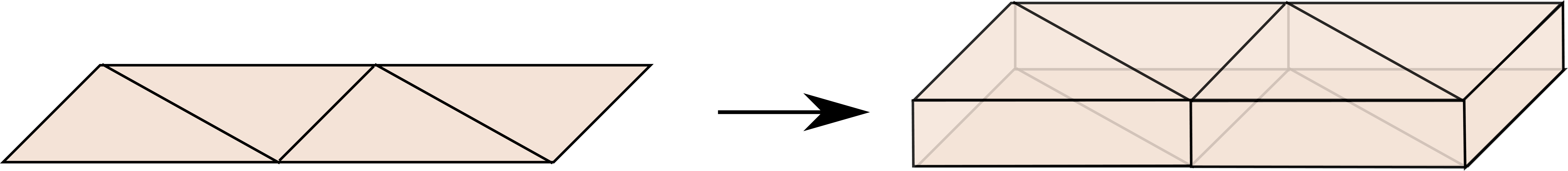%

\smallskip
a.
\smallskip

\def\svgwidth{11cm}
\executeiffilenewer{thicken3.svg}{thicken3.pdf}%
{inkscape -z -D --file=fig/thicken3.svg %
--export-pdf=fig/thicken3.pdf --export-latex}%
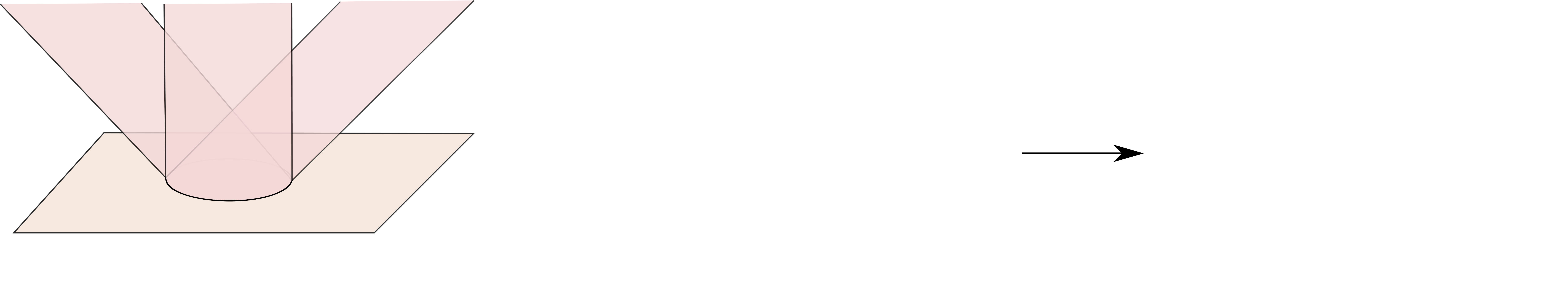%

\smallskip

\caption{a. The thickening of a surface. b.1. Four surfaces adjacent to a boundary curve 2. A sectional drawing of these and 3. A sectional drawing of their thickening.}
\label{F:thicken}
\end{figure}

\textbf{Third step.} In order to obtain a manifold without boundary, we \textit{double} $N$, that is, we consider the disjoint union of two copies $N_1$ and $N_2$ of $N$, and glue them along the boundary $\partial N_1=\partial N_2$ with the identity homeomorphism.

\smallskip

The following lemma shows that this construction can be computed in polynomial time.

\begin{lemma}\label{L:complexity}
A triangulation of the $3$-manifold $M$ can be computed in time
polynomial in $|I|=n+m$, the complexity of the initial \textsc{One-in-Three SAT} instance $I$.
\end{lemma}

\begin{proof}[Proof of Lemma~\ref{L:complexity}]
We first observe that the simplicial complex $K$ can be computed in
time polynomial in $|I|$: one can simply start with a big enough triangulation of the projective plane, remove disjoint triangles for
the branching circles, and glue triangulations of the surfaces
$F_{u_i}$ and $F_{\bar{u_i}}$ along these holes. The complexity of
this construction is clearly linear in $|I|$.

The thickening step first involves replacing triangles of $K$ by triangular
prisms (see Figure~\ref{F:thicken}a.) and retriangulating them, which is done in linear time. Then,
for every boundary curve, computing a triangulation of the thickening
pictured in Figure~\ref{F:thicken}b. can be done in time linear in the number of adjacent surfaces, which is bounded by $|I|$.

Finally, the doubling just amounts to taking two triangulations of $N$ and gluing them along their boundary. The complexity of this step is linear, and this concludes the proof.
\end{proof}

Finally, let us fix some notation for the rest of the section. There
is a natural projection $p: N \rightarrow K$ which corresponds to a
deformation retraction of the thickening (since it is by definition a
regular neighborhood). We define the continuous map $\tau: M
\rightarrow N$ as being the identity on $N_1$ and sending every point
of $N_2$ to its counterpart in $N_1$, and $\pi=\tau \circ p$.

\subsection{Proof of the reduction: the easy direction}

To prove Theorem~\ref{T:arnaud}, there remains to show how to build an
embedding of $S$ into $M$ from a satisfying assignment for $I$ and
vice-versa. The first direction is straightforward.

\begin{proposition}\label{P:main-easy} If there is a truth assignment
  for $I$ such that each clause in $C$ has exactly one true litteral,
  then $S$, the non-orientable surface of Euler genus $2m+2n+1$, embeds in $M$.
\end{proposition}

\begin{proof}
If there is a truth assignment for $I$ such that each clause in $C$
has exactly one true litteral, we can embed $S$ in $K$, and therefore
in $M$, in the following way. Take the union of $P$ and for every $i$,
either $F_{u_i}$ if $u_i$ is true, or $F_{\bar{u_i}}$, if $u_i$ is
false. Then exactly two boundary components are identified along each
boundary component of $P$, so we obtain a surface $S'$. Since $S'$ contains
$P$, it is non-orientable, and by construction $S'$ has Euler
genus $2m+2n+1$. Thus we have found an embedding of $S$.
\end{proof}

\subsection{Proof of the reduction: the hard direction}

The other direction will occupy us for the rest of the section.

\begin{proposition}\label{P:main} 
If $S$ embeds in $M$, then there is a truth assignment for $I$ such
that each clause in $C$ has exactly one true litteral.
\end{proposition}

\subparagraph{Outline of the proof.}
Proving this proposition is the main technical step of this section, and
it requires some tools from algebraic topology. Therefore, we first provide some intuition as to how the
proof goes.

\begin{wrapfigure}{r}{2.5cm}
\begin{center}
\begin{tikzcd}
S \arrow[hook]{r}{h} \arrow[swap]{dr}{f} & M \arrow{d}{\pi} \\
& K
\end{tikzcd}
\end{center}
\end{wrapfigure}

The natural idea would be to try to do the reverse of
Proposition~\ref{P:main-easy}, that is, starting from an embedding of
$S$ into $K$, to find the truth assignment by looking at which tube
the embedding chooses at every branching circle $u_i$. The difficulty
is that we do not start with an embedding into $K$, but only into
$M$. Composing this embedding $h$ with the map $\pi: M \rightarrow K$
leads to a continuous map $f:S \rightarrow K$, but $f$ has no reason
to be an embedding.

However, this approach can be salvaged. Following Agol, Hass and
Thurston~\cite{aht-cckgs-06}, we can still look at the
\emphdef{topological degree mod 2} induced by the continuous map $f$
at a point $x$ in $K$, which roughly counts the parity of how many
times $f$ maps $S$ to $x$. This number is constant where $K$ is a
surface, that is, outside of the branching circles $u_1, \ldots , u_n, c_1 \ldots , c_m$ of $K$, and the sum of the incoming degrees of the
patches of $K$ at a branching circle has to be $0$: intuitively, every
surface coming from one direction at a branching circle has to go
somewhere. Therefore, \textbf{if the degree of $f$ in $P$ is 1},
exactly one of the surfaces $F_{u_i}$ or $F_{\bar{u_i}}$ also has
degree 1. We can use it to define a truth assignment for the variable
in $U$, choosing $u_i$ to be true if $F_{u_i}$ has degree 1, and false
in the other case. Then, the sum of the degrees also has to be $0$ at
the circles corresponding to the clauses. Since $P$ has degree $1$,
this means that either one or three of the incoming surfaces also has
degree $1$. We show that this number is always one, otherwise the
surface $S$ can not have genus $2m+2n+1$. This will result from the
fact that if we have a degree 1 map between two surfaces $S_1$ and
$S_2$, then the genus of $S_2$ is not larger than the genus of $S_1$
(Lemma~\ref{L:deg1}). This shows that every clause has exactly one
true literal and concludes the proof.

This all hinges on the fact that the degree of $f$ in $P$ is one. This
is where our proof diverges from the one of Agol, Hass
and Thurston, as in their case this step is straightforward. Here,
this will result from the non-orientability of $S$: morally, when
embedding $S$ into $M$ and then mapping it into $K$, the only place
where the non-orientability can go is $P$. To prove this fact formally
is another matter and relies on three ingredients:

\begin{enumerate}
\item Since $S$ is non-orientable and has odd genus, in the image of the embedding $h(S) \subseteq M$, there is a non-trivial homology cycle $h(\alpha)$, which has order 2 in the $\mathbb{Z}$-homology of $M$ (Lemma~\ref{L:hempel}).
\item The kernel of the map $\pi_{\#}: H_1(M) \rightarrow H_1(K)$ has no torsion (Lemma~\ref{L:notorsion}). In particular, $\pi \circ h(\alpha)=f(\alpha)$ is non-trivial in the $\mathbb{Z}$-homology of $K$. Intuitively, the reason is that the $\mathbb{Z}_2$-subgroup of $H_1(M)$ comes from $P$, which is preserved by $\pi$. To prove this formally, we split $K$ and $M$ at the ``equator'' and exploit the naturality of the Mayer-Vietoris sequence (Lemma~\ref{L:notorsion}).
\item Using cup-products, which provide an algebraic bridge between (co-)homology in dimensions 1 and 2, we leverage on this to prove that $f$ has degree one on $P$.
\end{enumerate}

\smallskip

\subparagraph*{Introductory lemmas.} The notion of degree is conveniently expressed with the language of
homology. In the following, we will rely extensively on the following
notions: (relative) homology, Mayer-Vietoris sequence, cohomology,
Kronecker pairing (which we denote with brackets), cup-products, and
we will rely on Poincar\'e duality and the universal coefficient
theorem. Alas, introducing (or even defining) these falls widely
outside the scope of this paper, and we refer the reader to the
textbook of Hatcher~\cite{h-at-02} to get acquainted with these
concepts. For a map $f$, the induced maps in homology and cohomology
are respectively denoted by $f_{\#}$ and $f^{\#}$.

Let us first prove the three aforementioned lemmas. The first one
shows the non-triviality of maps from non-orientable surfaces of odd
genus to $3$-manifolds (see also Hempel~\cite[Lemma 5.1]{h-3m-04}). The
second one shows that the map $\pi$ only kills torsion-free elements
and the third one shows that degree 1 maps between surfaces can only
reduce the genus.  To streamline the notations, when no module is
indicated, homology and cohomology are taken with $\mathbb{Z}$
coefficients.

\begin{lemma}\label{L:hempel}
Let $S$ be a non-orientable surface of odd genus, and
$\alpha$ be a simple closed curve on $S$, inducing an element of order $2$ in $H_1(S)$. Let
$f:S\rightarrow M$ be an embedding of $S$ into a $3$-manifold
$M$. Then $f(\alpha)$ is not null-homologous in $H_1(M)$.
\end{lemma}

\begin{proof}
We recall that a co-dimension $1$ submanifold $M_1$ embedded in a
manifold $M_2$ is two-sided if its normal bundle is trivial, otherwise
it is one-sided. An embedded curve is orientation-preserving if it has
an orientable neighborhood, otherwise it is orientation-reversing.

Since $S$ has odd Euler characteristic, $\alpha$ is
orientation-reversing on $S$. Now we distinguish two cases: either $S$ is $2$-sided in $M$, or it is
$1$-sided. In the first case, $f(\alpha)$ is orientation reversing in
$f(S)$, and therefore also in $M$. Therefore it is non-trivial in
$\mathbb{Z}_2$ homology. In the second case, a small generic
perturbation of $f(\alpha)$ makes it have a single intersection point
with $f(S)$. By Poincar\'e duality with $\mathbb{Z}_2$-coefficients, it
is therefore non-trivial in $H_1(M, \mathbb{Z}_2)$. In both cases the
result follows by the universal coefficient theorem.

\end{proof}

\begin{lemma}\label{L:notorsion}
Let $K, M$ and $\pi$ be as introduced in Section~\ref{S:gadget}, then the kernel of the map $\pi_{\#}: H_1(M) \rightarrow H_1(K)$ has no torsion.
\end{lemma}

\begin{proof}
Let $K_1$ and $K_2$ denote the lower and the upper hemispheres of $K$ (see Figure~\ref{F:vietoris}), $M_1$ and $M_2$ be the corresponding subspaces of $M$. By naturality of the Mayer-Vietoris sequence with reduced homology, we obtain the following commutative diagram, where the horizontal lines are exact.

\begin{figure}
\centering
\def\svgwidth{14cm}
\executeiffilenewer{vietoris.svg}{vietoris.pdf}%
{inkscape -z -D --file=fig/vietoris.svg %
--export-pdf=fig/vietoris.pdf --export-latex}%
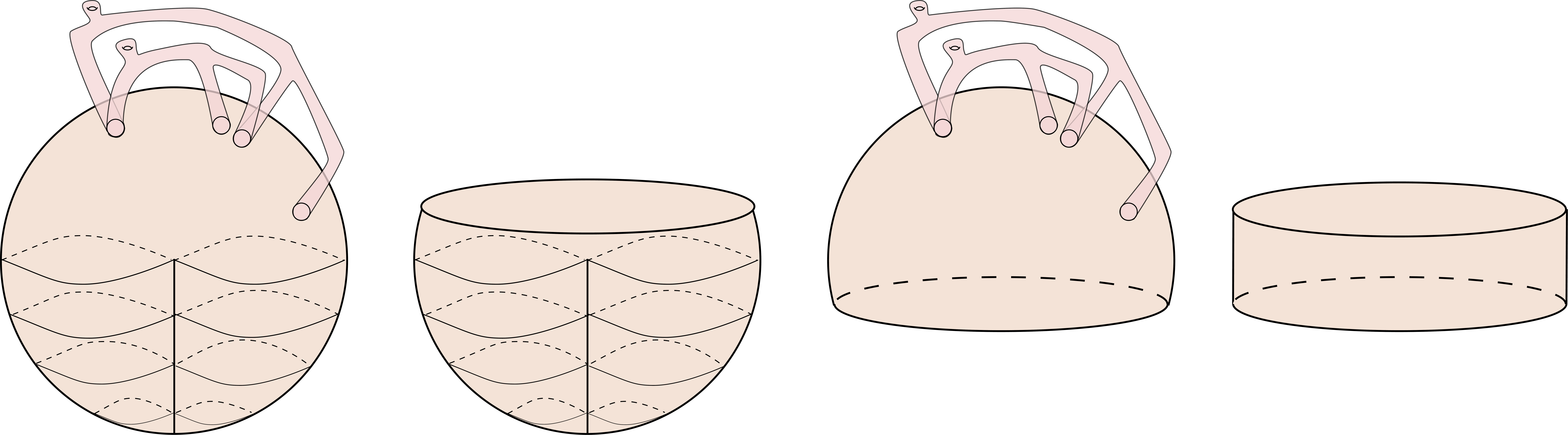%

\caption{Decomposing the complex $K$ around the equator.}
\label{F:vietoris}
\end{figure}

\begin{diagram}
 & \rTo & H_1(M_1 \cap M_2) &\rTo^{(i_{1},j_{1})} & H_1(M_1) \oplus H_1 (M_2) &\rTo^{k_1-l_1} & H_1(M)    &  \rTo^{\partial} & \widetilde{H_0}(M_1\cap M_2)&\rTo\\
   &    & \dTo_{\pi_{\#}}        &     &  \dTo_{(\pi_{1\#}, \pi_{2\#})}  &     & \dTo_{\pi_{\#}} & &  \dTo_{\pi_{\#}}\\
 & \rTo & H_1(K_1 \cap K_2) &\rTo^{(i_2,j_2)} & H_1(K_1) \oplus H_1(K_2) & \rTo^{k_2-l_2} & H_1(K) & \rTo^{\partial}&\widetilde{H_0}(K_1 \cap K_2)&\rTo
\end{diagram}

We first remark that when applied to a surface, the process of
thickening and doubling amounts to taking the product with
$S^1$. Therefore, we know that $K_2$ is a M\"obius band, $M_2$ is a
M\"obius band times a circle, $K_1 \cap K_2$ retracts to $S^1$ and
$M_1 \cap M_2$ retracts to a torus $T$. Since $M_1 \cap M_2$ and $K_1
\cap K_2$ are connected, their reduced 0-homologies are zero, thus the
maps $k_1-l_1$ and $k_2-l_2$ are surjective. Furthermore, $\pi_{2\#}$
is the projection of the $S^1$ fiber, $\Ima(i_1)=\mathbb{Z}^2$,
$\Ima(j_1)=\mathbb{Z} \oplus 2\mathbb{Z}$, $\Ima(i_2)=\mathbb{Z}$ and
$\Ima(j_2)=2 \mathbb{Z}$. Therefore, $k_1(H_1(M_1) \oplus H_1(M_2))$
contains no torsion, and the torsion subgroup of $H_1(M)$ comes from
$-l_1(H_1(M_2))$. Similarly, the torsion subgroup of $H_1(K)$ comes
from $-l_2(H_1(K_2))$. By commutativity of the diagram,
$(k_2-l_2)\circ (\pi_{\#1},\pi_{2\#})=\pi_{\#} \circ (k_1-l_1)$ and
thus their image contains the $\mathbb{Z}_2$ torsion subgroup of
$H_1(K)$. Therefore, the kernel of the map $\pi_{\#}$ has no torsion
and the claim is proved.

\end{proof}

\begin{lemma}\label{L:deg1}
Let $f:S_1 \rightarrow S_2$ a continuous map of degree one mod 2 between
two surfaces $S_1$ and $S_2$. Then the genus of $S_2$ is not larger than the genus of $S_1$.
\end{lemma}

\begin{proof}[Proof of Lemma~\ref{L:deg1}]
Let us denote by $g_1$ the genus of $S_1$ and by $g_2$ the genus of
$S_2$. Assume by
contradiction that $g_2$ is larger than $g_1$. The map $f$ induces a
map on first cohomology groups
$f^{\#}:H^1(S_2,\mathbb{Z}_2)=\mathbb{Z}_2^{g_2} \rightarrow
H^1(S_1,\mathbb{Z}_2)=\mathbb{Z}_2^{g_1}$. Since $g_2$ is larger than
$g_1$, by dimension this map has a non-trivial kernel. Pick an element
$a$ in this kernel, and an element $b\in H^1(S_2,\mathbb{Z}_2)$ such
that $a \cup b$ generates $H^2(S_2,\mathbb{Z}_2)$ (which exists by
Poincar\'e duality). Then $f^{\#}(a \cup b)=0$ by naturality of the
cup product.  But then, by naturality of the Kronecker pairing, we
have

\[ 0=\langle f^{\#}(a \cup b), [S_1] \rangle = \langle a \cup b, f_{\#}([S_1]) \rangle = \langle a \cup b , [S_2]\rangle = 1, \]
and we have reached a contradiction.

\end{proof}

\subparagraph*{Wrapping up the proof.}

We can now proceed with the proof.

\begin{proof}[Proof of Proposition~\ref{P:main}]
Let us denote by $h$ the embedding from $S$ into $M$, and by $\alpha$
a simple cycle of order $2$ (in homology over $\mathbb{Z}$) in $S$. By
Lemma~\ref{L:hempel}, $h(\alpha)$ is not null-homologous in $M$, and
it has order $2$ in $H_1(M)$. By Lemma~\ref{L:notorsion}, $\pi_{\#}:H_1(M) \rightarrow H_1(K)$ does not have $h(\alpha)$ in its kernel, therefore we obtain that $\pi \circ h(\alpha)=f(\alpha)$ is not null-homologous in $K$, and since $\alpha$ has order $2$ in $H_1(S)$, it also has order $2$ in $H_1(K)$.

But there is a unique homology class of order $2$ in $H_1(K)$, which
is the one induced by the simple cycle $\beta$ which has order $2$ in
$P$. Therefore $h(\alpha)$ is homologous to $\beta$.

We now switch to $\mathbb{Z}_2$
coefficients, in order to use the 2-dimensional homology despite the non-orientability. Since $\mathbb{Z}_2$ is a field, homology and cohomology
with $\mathbb{Z}_2$ coefficients are dual to each other so we can take
a cohomology class $b$ in $H^1(K, \mathbb{Z}_2)$ which evaluates to
$1$ on $[\beta]$. The map $f^{\#} :
H^1(K,\mathbb{Z}_2) \rightarrow H^1(S,\mathbb{Z}_2)$ maps $b$ to a
cohomology class $a \in H^1(S,\mathbb{Z}_2)$, and by naturality of the
Kronecker pairing, we have \[\langle a,[\alpha] \rangle = \langle
f^{\#}(b),[\alpha]\rangle = \langle b,f_{\#}([\alpha]) \rangle=
\langle b,[\beta] \rangle=1,\] where the brackets denote taking the
representative in 1-dimensional homology with $\mathbb{Z}_2$
coefficients and the last equality follows from the definition of
$b$. Now, denote by $(\alpha, \beta_1, \gamma_1, \beta_2, \gamma_2,
\ldots \beta_k, \gamma_k)$ a family of simple curves forming a basis
of $H_1(S, \mathbb{Z}_2)$, such that each pair $(\beta_i,\gamma_i)$
intersects once and there are no other intersections. We have that $a$ is the Poincar\'e dual of $\alpha + \sum_I
\beta_i + \sum_J \gamma_j$, for some subsets $I,J \subseteq
     [k]$, and since $\langle a, [\alpha] \rangle=1$, a quick computation in the ring $H^*(S,\mathbb{Z}_2)$ shows that $a \cup a=\xi$,
     where $\xi$ is the generator of $H^2(S, \mathbb{Z}_2)$.

Now, by naturality of the cup-product, we obtain $f^{\#}(b \cup b)=f^{\#}(b) \cup f^{\#}(b)=a \cup a=\xi$. Furthermore,
once again by naturality of the Kronecker pairing, and writing $[S]$
for the fundamental class mod 2 of $S$, we have
\begin{equation} 1= \langle \xi, [S] \rangle= \langle f^{\#}(b \cup b), [S] \rangle
= \langle b\cup b,f_{\#}([S])\rangle. \label{E:main} \end{equation}

Let us open a parenthesis and recall how the notion of \emphdef{degree} of a continuous
map can be extended when the target is not a manifold, applied to our specific case. The map $f$ induces a mapping $f_{\#}$ in \textit{relative homology}
between $H_2(S,\emptyset,\mathbb{Z}_2)$ and $H_2(K,B,\mathbb{Z}_2)$,
where $B$ is the set of branching circles of $K$. The group
$H_2(K,B,\mathbb{Z}_2)$ is generated by the homology classes induced
by the \textit{pieces} $P, F_{u_i}$ and $F_{\bar{u_i}}$, and therefore
the image of $f_{\#}(S)$ associates to each piece a 0 or 1 number, the
\textit{topological degree mod 2} of $f$ on this piece. An equivalent
view of this number is the following. By standard transversality
arguments, the map $f:S\rightarrow K$ can be homotoped so as to be a
union of homeomorphisms of subsurfaces of $S$ into one of the pieces
$P, F_{u_i}, F_{\bar{u_i}}$ forming $K$. The parity of the number of
subsurfaces of $S$ mapped to a piece $P, F_{u_i}$ or $F_{\bar{u_i}}$
is also the topological degree mod 2 of the map $f$. This second point of view shows that the sum of the degrees of the pieces adjacent to a branching circle is $0$, as $S$ has no boundary.

Going back to the proof, we observe that, juggling between both interpretations of the degree, the geometric meaning of Equation~\eqref{E:main} is that $f(S)$ covers the
intersection point of two perturbated copies of $\beta$ an odd number
of times, and as this intersection point is in $P$, the topological degree mod 2 of $f$ on $P$ is $1$.

The sum of the incoming degrees of $2$-dimensional patches along a
boundary curve $u_i$ or $c_i$ in $K$ is $0$. Therefore, around every
boundary curve $u_i$, this allows us to pick a truth assignment for
$u_i$, depending on whether $f$ has degree $1$ on $F_{u_i}$ or
$F_{\bar{u_i}}$. This will conclude the proof if we prove that this
truth assignment $\varphi$ is valid for the $1$-in-$3$ SAT instance $|I|$.

For every clause $c_i$, there are exactly $4$ surfaces adjacent to the
boundary curve $c_i$, one of these being $P$, and we denote the others
by $F_1$, $F_2$ and $F_3$. Since $f$ has degree $1$
on $P$, it has degree $1$ either on one of the other surfaces or on
all three. If we are in the former case for every clause, this shows
that all the clauses are
satisfied exactly by one of its variables under the truth assignment
$\varphi$, and we are done. Otherwise, for every clause where $f$ has
degree one on all three surfaces $F_1$, $F_2$ and $F_3$, pick arbitrarily one, say $F_1$, and consider the surface $S'$
obtained by gluing every such $F_1$ to $P$ and every $F_2$ and $F_3$
together. We claim that this surface has genus strictly larger than
$2n+2m+1$:

\begin{itemize}
\item The projective plane $P$ contributes by $1$.
\item For every $i$, exactly one of the surfaces $F_{u_i}$ or $F_{\bar{u_i}}$ is chosen. Since they have (Euler) genus two, they contribute by $2$.
\item For every clause, the gluing of $F_1$ to $P$ increases the genus by $2$. We have already reached $2n+2m+1$.
\item Every time we glue $F_2$ and $F_3$ together, we increase the genus yet again.
\end{itemize}

But by definition of $S'$, there is a degree one map from $S$ to
$S'$, which is impossible by Lemma~\ref{L:deg1}. This concludes the proof.
\end{proof}

The combination of Lemma~\ref{L:complexity} and Proposition~\ref{P:main} provides
a polynomial reduction from 1-in-3 SAT to the problem of deciding the
embeddability of a non-orientable surface into a $3$-manifold, which
concludes the proof of Theorem~\ref{T:arnaud}.


\section{An algorithm to find non-orientable surfaces of odd Euler genus in 3-manifolds}\label{S:normal}

In this section we prove the following theorem.

\Tnormal*

We first observe that if a non-orientable surface $S$ of genus $g$ embeds in a $3$-manifold $M$, then all the non-orientable surfaces of genus $g+2k$ for $k > 1$ also embed into $M$, since one can add orientable handles in a small neighborhood of $S$. Therefore, to prove Theorem~\ref{T:normal} it is enough to find the non-orientable surface of \textbf{minimal} odd Euler genus which embeds into $M$, and this is what our algorithm will do.

Let us also note that if $M$ is non-orientable, it contains a solid Klein
bottle in the neighborhood of an orientation-reversing
curve. Therefore it also contains every non-orientable surface of even
genus, and the algorithm is trivial in this case. Thus, the only case not covered by our algorithm is the one of non-orientable surfaces of even Euler characteristic in orientable manifolds.

\subsection{Background on low-dimensional topology and normal surfaces}

We introduce here quickly the tools we are using from $3$-dimensional topology and normal surfaces, and refer to Hass, Lagarias and Pippenger~\cite{hlp-ccklp-99} or Matveev~\cite{m-atc3m-03} for more background.

A $3$-manifold $M$ is \emphdef{irreducible} if every sphere embedded
in $M$ bounds a ball in $M$.  The \emphdef{connected sum} $M_1 \# M_2$ of two
$3$-manifolds $M_1$ and $M_2$ is obtained by removing a small ball
from both $M_1$ and $M_2$ and gluing together the resulting boundary
spheres. A $3$-manifold is \emphdef{prime} if it can not be presented as a
connected sum of more than one manifold, none of which is a
sphere. It is well known~\cite[Proposition~1.4]{h-nb3mt-07} that prime
manifolds are irreducible, except for $S^2 \times S^1$ and the
non-orientable bundle $S^2 \tilde{\times} S^1$.

Let $S$ be a surface embedded in $M$. A \emphdef{compressing disk} for $S$ is an embedded disk $D \subset M$ whose interior is disjoint from $S$ and whose boundary is a non-contractible loop in $S$. A surface is \emphdef{compressible} if it has a compressing disk and \emphdef{incompressible} if not. If a surface $S$ is compressible, one can cut it along the boundary of a compressing disk and glue disks on the resulting boundaries, this reduces its genus by $2$.

To introduce normal surfaces, we denote by $T$ a triangulation of a 3--manifold~$M$. A \emphdef{normal
  isotopy} is an ambient isotopy of $M$ that is fixed on the 2--skeleton
of $T$. A \emphdef{normal surface} in~$T$ is a properly embedded surface
in~$T$ that meets each tetrahedron in a (possibly empty) disjoint
collection of \emphdef{normal disks}, each of which is either a
\emphdef{triangle} (separating one vertex of the tetrahedron from the other
three) or a \emphdef{quadrilateral} (separating two vertices from the other
two).  In each tetrahedron, there are $4$ possible types of triangles and
$3$ possible types of quadrilaterals, pictured in Figure~\ref{F:Normaldisks}.  

\begin{figure}[htb]
\centering
\def\svgwidth{10cm}
\executeiffilenewer{Normaldisks.svg}{Normaldisks.pdf}%
{inkscape -z -D --file=fig/Normaldisks.svg %
--export-pdf=fig/Normaldisks.pdf --export-latex}%
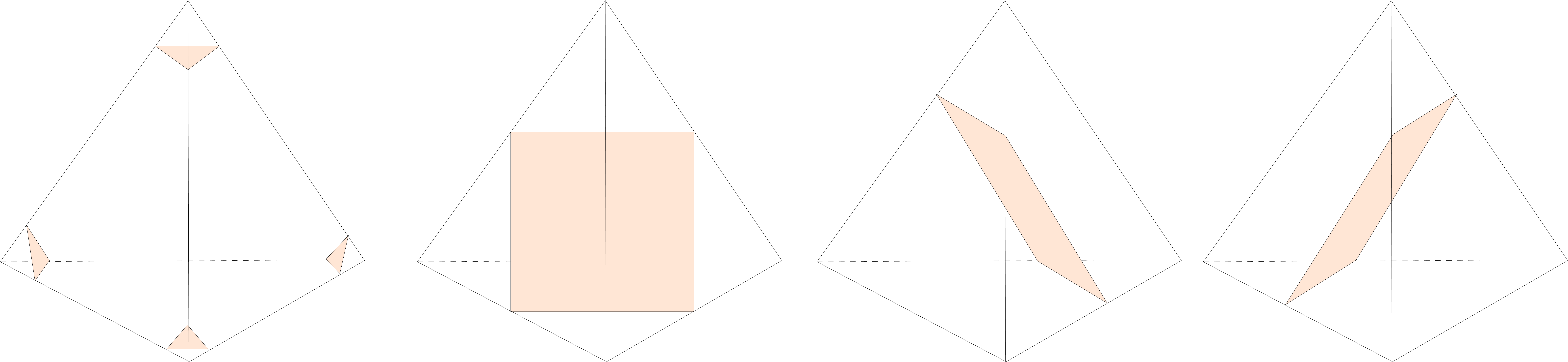%

\caption{The seven types of normal disks within a given tetrahedron: Four
  triangles and three quadrilaterals.}
\label{F:Normaldisks}
\end{figure}

Normal surfaces are used to investigate combinatorially and computationally the surfaces embedded in a $3$-manifold. In this endeavour, the first step is to prove that the surfaces we are interested in can be \emphdef{normalized}, that is, represented by normal surfaces. The following theorem is due to Haken~\cite[Chapter 5]{h-tnik-61}, we refer to the book of Matveev for a proof.

\begin{theorem}[{\cite[Corollary~3.3.25]{m-atc3m-03}}]\label{T:matveev1}
Let $M$ be an irreducible $3$-manifold and $S$ be an incompressible surface embedded in $M$. Then, if $S$ is not a sphere, it is ambient isotopic to a normal surface.
\end{theorem}

For $S$ a normal surface, denote by $e(S)$ the \emphdef{edge degree}
of $S$, that is, the number of intersections of $S$ with the
$1$-skeleton of the triangulation $T$. A normal surface is
\emphdef{minimal} if it has minimal edge degree over all the normal
surfaces isotopic to it. Each embedded normal surface has associated
\emphdef{normal coordinates}: a vector in $\mathbb{Z}_{\geq 0}^{7t}$,
where $t$ is the number of tetrahedra in $T$, listing the number of
triangles and quadrilaterals of each type in each tetrahedron. These
coordinates provide an algebraic structure to normal surfaces: there
is a one-to one correspondence between normal surfaces up to normal
isotopy and normal coordinates satisfying some constraints, called the
matching equations and the quadrilateral constraints. In particular, one can add
normal surfaces by adding their normal coordinates, this is called a
\emphdef{Haken sum}. Among normal surfaces, ones of particular
interest are the \emphdef{fundamental} normal surfaces, which are
surfaces that can not be written as a sum of other non-empty normal
surfaces. Every normal surface can be decomposed as a sum of
fundamental normal surfaces, and the following theorem provides tools
to understand these.

\begin{theorem}[{\cite[Corollary~4.1.37]{m-atc3m-03}, see also Jaco and Oertel~\cite{jo-ad3mhm-84}}]\label{T:matveev2}
Let a minimal connected normal surface $S$ in an irreducible $3$-manifold $M$ be presented as a sum $S=\sum_{i=1}^nS_i$ of $n >1$ nonempty normal surfaces. If $S$ is incompressible, so are the $S_i$. Moreover, no $S_i$ is a sphere or a projective plane.
\end{theorem}

\subsection{Crushing}\label{S:crushing}

In order to rely on normal surface theory and apply the aforementioned
theorems, we would like $M$ to be irreducible. Therefore, the first
step of the algorithm is to to simplify the $3$-manifold $M$ so as to
make it irreducible. In order to do this, we rely on the operation of
\emphdef{crushing}, which was introduced by Jaco and
Rubinstein~\cite{jr-0et3m-03}, and extended to the non-orientable case
(as well as simplified) by Burton~\cite{b-nac3mt-14}. In particular,
Burton proves the following theorem~\cite[Algorithm~7]{b-nac3mt-14}.

\begin{theorem}\label{T:burton}
  Given a $3$-manifold $M$, there is an algorithm which either
  decomposes $M$ into a connected sum of prime manifolds, or else
  proves that $M$ contains an embedded two-sided projective plane.

Furthermore, this algorithm is in \NP \, in the following sense: there
exists a certificate of polynomial size (namely, the list of
fundamental normal surfaces along which to crush) allowing to compute
in polynomial time the triangulations of the summands or output that
$M$ contains an embedded two-sided projective plane.
\end{theorem}

If this algorithm outputs an embedded projective plane, we are done, since in this case our $3$-manifold $M$ contains every non-orientable surface of odd genus. If not, if we are provided the aforementioned certificate we can proceed separately on every summand, thanks to the following easy lemma.

\begin{lemma}\label{L:connecsum}
Let $M$ be a connected sum of $3$-manifolds $M_1, \ldots M_k$. Then if a non-orientable surface $S$ of odd genus $g$ embeds into $M$, it also embeds into one of the $M_i$.
\end{lemma}

\begin{proof}
The $3$-manifolds $M_i$ are obtained from $M$ after a cut-and-paste procedure along $2$-spheres, and let $S_i$ denote the surface $S$ obtained in $M_i$ after this procedure. The surface $S$ is the connected sum of the surfaces $S_i$, and since $g$ is odd and the Euler characteristic is additive under connected sums, one of the $S_i$ has odd Euler characteristic. In particular it is non-orientable. By adding orientable handles to this $S_i$ inside $M_i$, one obtains a non-orientable surface of the same genus as $S$, and therefore an embedding of $S$.
\end{proof}

If one of the summands is prime but not irreducible, then, as mentioned before, it is homeomorphic either to $S^2 \times S^1$ or the twisted bundle $S^2 \tilde{\times} S^1$. One of the features~\cite[Algorithm 7]{b-nac3mt-14} of the crushing algorithm that we use is that the $S^2 \times S^1$ and $S^2 \tilde{\times} S^1$ summands in the prime decomposition are actually rebuilt afterwards based on the homology of the input $3$-manifold. In particular, we know precisely if there are any and how many of them there are, without having to use some hypothetical recognition algorithm. Furthermore, the following lemma shows that these summands are uninteresting for our purpose.

\begin{lemma}
No non-orientable surface of odd genus embeds into $S^2 \times S^1$ or $S^2 \tilde{\times} S^1$.
\end{lemma}

\begin{proof}
By Lemma~\ref{L:hempel}, if such an embedding existed, there would be an element of order $2$ in $H_1(S^2 \times S^1)$ or $H_1(S^2 \tilde{\times} S^1)$, which is a contradiction since both of these groups are equal to $\mathbb{Z}$. 
\end{proof}

Therefore, the output of our algorithm is trivial for these summands, and in the rest of this section we assume that the manifold $M$ is irreducible.

\subsection{Fundamental normal surfaces}

We now show that in order to find the non-orientable surface of minimal odd genus, it is enough to look at the fundamental normal surfaces.

\begin{proposition}\label{P:fundamental}
If a non-orientable surface of minimal odd genus embeds in an irrreducible
$3$-manifold $M$, then it is witnessed by one of the fundamental
normal surfaces. If none of the fundamental normal surfaces have odd
genus, then no surface of odd genus embeds into $M$.
\end{proposition}

Before proving this proposition, let us show how it implies Theorem~\ref{T:normal}.

\begin{proof}[Proof of Theorem~\ref{T:normal}]
  By applying the crushing procedure and following Lemma~\ref{L:connecsum} and the discussion in Section~\ref{S:crushing}, one can assume that $M$ is irreducible if one is given the certificate of Theorem~\ref{T:burton}. Then, by Proposition~\ref{P:fundamental}, the non-orientable surface of minimal odd genus, if it exists, appears among one of the fundamental normal surfaces. By a now standard argument of Hass, Lagarias and Pippenger~\cite[Lemma 6.1]{hlp-ccklp-99}, the coordinates of fundamental normal surfaces can be described with a polynomial number of bits. Since there are $7t$ coordinates for a triangulation of size $T$, we can therefore use this as a second half of the \NP \, certificate. Now, if the input genus $g$ is at least the minimal one witnessed by this certificate, then the non-orientable surface of genus $g$ is embeddable in $M$, otherwise it is not.
\end{proof}

We now prove Proposition~\ref{P:fundamental}.

\begin{proof}[Proof of Proposition~\ref{P:fundamental}]
  Let $S$ be a surface of minimal odd genus $g$ embedded in $M$. We first claim that $S$ is incompressible. Indeed, if it is not, let $D$ be a compressing disk and $S \mid D$ be the surface obtained after the compression along $D$: then $S \mid D$ has genus $g-2$ which contradicts the minimality of $g$.

  The surface $S$ being incompressible, then by Theorem~\ref{T:matveev1}, there exists a normal surface isotopic to it. Let us denote by $S'$ a normal surface of genus $g$ and of minimal edge degree among all of those. If $S'$ is not fundamental, by Theorem~\ref{T:matveev2}, then it can be written as a sum of fundamental normal surfaces $S'=\sum_{i=1}^n S_i$ such that the $S_i$ are incompressible and none of them are spheres or projective planes. In particular, none of the surfaces $S_i$ have positive Euler characteristic. Since the Euler characteristic is additive on the space of normal coordinates, one of the surfaces $S_i$ has odd genus at most $g$. By minimality of $g$, this surface $S_i$ actually has genus $g$, and it has smaller edge degree by $S'$, which is a contradiction. Therefore $S'$ is fundamental, which concludes the proof.

\end{proof}

\textbf{Remark:} The reason why the above proof fails in the case of even genus is that in general a non-orientable surface of genus $g$ might be written as a Haken sum of orientable surfaces. In our case, this issue is avoided by the fact that a surface of odd Euler genus is necessarily non-orientable. For even Euler genus, the first problem that we do not solve is the one of deciding whether a given $3$-manifold contains a Klein bottle. For this specific case, we believe that the problem should be decidable, by computing a JSJ decomposition and identifying in the geometric pieces which ones contain Klein bottles: hyperbolic pieces do not, and one can detect which Seifert fibered spaces do just based on their invariants. However, this technique does not seem to apply to higher genera.

\subparagraph*{Acknowledgements}
We would like to thank Saul Schleimer and Eric Sedgwick for stimulating discussions, and the anonymous reviewers for helpful comments.


\bibliographystyle{plain}
\bibliography{bibexport,biblio}

\end{document}

%% file: BranchSurf.pdf_tex
\begingroup%
  \makeatletter%
  \providecommand\color[2][]{%
    \errmessage{(Inkscape) Color is used for the text in Inkscape, but the package 'color.sty' is not loaded}%
    \renewcommand\color[2][]{}%
  }%
  \providecommand\transparent[1]{%
    \errmessage{(Inkscape) Transparency is used (non-zero) for the text in Inkscape, but the package 'transparent.sty' is not loaded}%
    \renewcommand\transparent[1]{}%
  }%
  \providecommand\rotatebox[2]{#2}%
  \ifx\svgwidth\undefined%
    \setlength{\unitlength}{1126.92922573bp}%
    \ifx\svgscale\undefined%
      \relax%
    \else%
      \setlength{\unitlength}{\unitlength * \real{\svgscale}}%
    \fi%
  \else%
    \setlength{\unitlength}{\svgwidth}%
  \fi%
  \global\let\svgwidth\undefined%
  \global\let\svgscale\undefined%
  \makeatother%
  \begin{picture}(1,1.24177133)%
    \put(0,0){\includegraphics[width=\unitlength,page=1]{BranchSurf.pdf}}%
    \put(0.21728674,0.67726161){\color[rgb]{0,0,0}\makebox(0,0)[b]{\smash{$\ldots$}}}%
    \put(0,0){\includegraphics[width=\unitlength,page=2]{BranchSurf.pdf}}%
    \put(0.71311285,0.68596864){\color[rgb]{0,0,0}\makebox(0,0)[b]{\smash{$\ldots$}}}%
    \put(0.38797749,0.82676652){\color[rgb]{0,0,0}\makebox(0,0)[b]{\smash{$u_1$}}}%
    \put(0.3121089,0.75778074){\color[rgb]{0,0,0}\makebox(0,0)[b]{\smash{$u_2$}}}%
    \put(0.18031419,0.5691023){\color[rgb]{0,0,0}\makebox(0,0)[b]{\smash{$u_n$}}}%
    \put(0.5329696,0.83463349){\color[rgb]{0,0,0}\makebox(0,0)[b]{\smash{$c_1$}}}%
    \put(0.59638026,0.76355972){\color[rgb]{0,0,0}\makebox(0,0)[b]{\smash{$c_2$}}}%
    \put(0.73828934,0.58676273){\color[rgb]{0,0,0}\makebox(0,0)[b]{\smash{$c_m$}}}%
    \put(0,0){\includegraphics[width=\unitlength,page=3]{BranchSurf.pdf}}%
    \put(0.4731114,1.20941016){\color[rgb]{0,0,0}\makebox(0,0)[b]{\smash{$F_{\bar{u_1}}$}}}%
    \put(0.47600254,1.08942932){\color[rgb]{0,0,0}\makebox(0,0)[b]{\smash{$F_{u_1}$}}}%
    \put(0.92365176,0.18602624){\color[rgb]{0,0,0}\makebox(0,0)[lb]{\smash{$P$}}}%
    \put(0,0){\includegraphics[width=\unitlength,page=4]{BranchSurf.pdf}}%
  \end{picture}%
\endgroup%

%% file: thicken1.pdf_tex
\begingroup%
  \makeatletter%
  \providecommand\color[2][]{%
    \errmessage{(Inkscape) Color is used for the text in Inkscape, but the package 'color.sty' is not loaded}%
    \renewcommand\color[2][]{}%
  }%
  \providecommand\transparent[1]{%
    \errmessage{(Inkscape) Transparency is used (non-zero) for the text in Inkscape, but the package 'transparent.sty' is not loaded}%
    \renewcommand\transparent[1]{}%
  }%
  \providecommand\rotatebox[2]{#2}%
  \ifx\svgwidth\undefined%
    \setlength{\unitlength}{3508.54400121bp}%
    \ifx\svgscale\undefined%
      \relax%
    \else%
      \setlength{\unitlength}{\unitlength * \real{\svgscale}}%
    \fi%
  \else%
    \setlength{\unitlength}{\svgwidth}%
  \fi%
  \global\let\svgwidth\undefined%
  \global\let\svgscale\undefined%
  \makeatother%
  \begin{picture}(1,0.10903194)%
    \put(0,0){\includegraphics[width=\unitlength,page=1]{thicken1.pdf}}%
  \end{picture}%
\endgroup%

%% file: thicken3.pdf_tex
\begingroup%
  \makeatletter%
  \providecommand\color[2][]{%
    \errmessage{(Inkscape) Color is used for the text in Inkscape, but the package 'color.sty' is not loaded}%
    \renewcommand\color[2][]{}%
  }%
  \providecommand\transparent[1]{%
    \errmessage{(Inkscape) Transparency is used (non-zero) for the text in Inkscape, but the package 'transparent.sty' is not loaded}%
    \renewcommand\transparent[1]{}%
  }%
  \providecommand\rotatebox[2]{#2}%
  \ifx\svgwidth\undefined%
    \setlength{\unitlength}{7621.59869264bp}%
    \ifx\svgscale\undefined%
      \relax%
    \else%
      \setlength{\unitlength}{\unitlength * \real{\svgscale}}%
    \fi%
  \else%
    \setlength{\unitlength}{\svgwidth}%
  \fi%
  \global\let\svgwidth\undefined%
  \global\let\svgscale\undefined%
  \makeatother%
  \begin{picture}(1,0.1830545)%
    \put(0,0){\includegraphics[width=\unitlength,page=1]{thicken3.pdf}}%
    \put(0.13544077,0.0004787){\color[rgb]{0,0,0}\makebox(0,0)[lb]{\smash{b.1.}}}%
    \put(0,0){\includegraphics[width=\unitlength,page=2]{thicken3.pdf}}%
    \put(0.48742299,0.00005945){\color[rgb]{0,0,0}\makebox(0,0)[lb]{\smash{b.2.}}}%
    \put(0,0){\includegraphics[width=\unitlength,page=3]{thicken3.pdf}}%
    \put(0.89059095,0.0004787){\color[rgb]{0,0,0}\makebox(0,0)[lb]{\smash{b.3.}}}%
  \end{picture}%
\endgroup%

%% file: vietoris.pdf_tex
\begingroup%
  \makeatletter%
  \providecommand\color[2][]{%
    \errmessage{(Inkscape) Color is used for the text in Inkscape, but the package 'color.sty' is not loaded}%
    \renewcommand\color[2][]{}%
  }%
  \providecommand\transparent[1]{%
    \errmessage{(Inkscape) Transparency is used (non-zero) for the text in Inkscape, but the package 'transparent.sty' is not loaded}%
    \renewcommand\transparent[1]{}%
  }%
  \providecommand\rotatebox[2]{#2}%
  \ifx\svgwidth\undefined%
    \setlength{\unitlength}{4813.149899bp}%
    \ifx\svgscale\undefined%
      \relax%
    \else%
      \setlength{\unitlength}{\unitlength * \real{\svgscale}}%
    \fi%
  \else%
    \setlength{\unitlength}{\svgwidth}%
  \fi%
  \global\let\svgwidth\undefined%
  \global\let\svgscale\undefined%
  \makeatother%
  \begin{picture}(1,0.27706822)%
    \put(0,0){\includegraphics[width=\unitlength]{vietoris.pdf}}%
    \put(0.22525837,0.02551493){\color[rgb]{0,0,0}\makebox(0,0)[b]{\smash{$K$}}}%
    \put(0.49448768,0.02595805){\color[rgb]{0,0,0}\makebox(0,0)[b]{\smash{$K_2$}}}%
    \put(0.89529888,0.02595805){\color[rgb]{0,0,0}\makebox(0,0)[b]{\smash{$K_1\cap K_2$}}}%
    \put(0.62888142,0.02595805){\color[rgb]{0,0,0}\makebox(0,0)[b]{\smash{$K_1$}}}%
  \end{picture}%
\endgroup%

%% file: Normaldisks.pdf_tex
\begingroup%
  \makeatletter%
  \providecommand\color[2][]{%
    \errmessage{(Inkscape) Color is used for the text in Inkscape, but the package 'color.sty' is not loaded}%
    \renewcommand\color[2][]{}%
  }%
  \providecommand\transparent[1]{%
    \errmessage{(Inkscape) Transparency is used (non-zero) for the text in Inkscape, but the package 'transparent.sty' is not loaded}%
    \renewcommand\transparent[1]{}%
  }%
  \providecommand\rotatebox[2]{#2}%
  \ifx\svgwidth\undefined%
    \setlength{\unitlength}{5454.25bp}%
    \ifx\svgscale\undefined%
      \relax%
    \else%
      \setlength{\unitlength}{\unitlength * \real{\svgscale}}%
    \fi%
  \else%
    \setlength{\unitlength}{\svgwidth}%
  \fi%
  \global\let\svgwidth\undefined%
  \global\let\svgscale\undefined%
  \makeatother%
  \begin{picture}(1,0.23057306)%
    \put(0,0){\includegraphics[width=\unitlength]{Normaldisks.pdf}}%
  \end{picture}%
\endgroup%